\documentclass[a4paper,11pt,english,oneside]{article}

\usepackage[latin1]{inputenc}

\usepackage{calc}
\usepackage{graphicx,palatino,verbatim}

\usepackage{amsmath}

\usepackage[T1]{fontenc}

\usepackage[english]{babel}

\usepackage{newlfont}

\usepackage{amsthm}

\usepackage{amssymb}

\usepackage{amsmath}

\usepackage{eufrak}

\usepackage{latexsym}

\usepackage{young}
\usepackage[mathscr]{eucal}

\usepackage{amsfonts}
\usepackage{syntonly}
\usepackage{graphicx}
\usepackage{mathptmx}
\usepackage{subfig}
\usepackage{url}
\usepackage{authblk}

\usepackage[letterpaper]{geometry}
\input xy
\xyoption{all}

\newtheorem{theorem}{Theorem}[section]
\newtheorem{remark}[theorem]{Remark}
\newtheorem{lemma}[theorem]{Lemma}
\newtheorem{definition}[theorem]{Definition}

\newtheorem{proposition}[theorem]{Proposition}

\newtheorem{corollary}[theorem]{Corollary}

\newtheorem{assumption}{Assumption}

\newtheorem{algorithm}{Algorithm}

\newcommand{\FF}{\mathbb F}

\def\cC{\mathcal C}

\def\deg{\mbox{\rm deg}}

\def\Im{\mbox{\rm Im}}

\def\dim{\mbox{\rm dim}}

\def\f2{{\mathbb F}_{2}}
\def\fm{{\mathbb F}_{2^m}}

\newcommand{\AGL}{\mbox{\rm AGL}}

\newcommand{\GL}{\mbox{\rm GL}}

\newcommand{\ga}{\alpha}

\newcommand{\gd}{\delta}

\newcommand{\gl}{\lambda}
\newcommand{\gk}{\kappa}

\newcommand{\gs}{\sigma}

\def\Im{\mathrm{Im}}

\newcommand{\keywords}[1]{%
  \begingroup 
  \renewcommand\subsubsection{}%
  \addtocounter{subsubsection}{-1}%
  \subsubsection{\small{\textbf{Keywords}:} #1}%
  \endgroup}

\begin{document}
\date{}

\setcounter{Maxaffil}{1}
\title{\LARGE The role of Boolean functions in hiding sums as trapdoors for some block ciphers}
\author{\rm R. Aragona\footnote {riccardo.aragona@unitn.it, $ ^{**}$marco.calderini@unitn.it, $ ^{\#}$maxsalacodes@gmail.com } }
\author[$**$\, 1]{\rm  M. Calderini}
\author[$\#$\, 1]{\rm M. Sala}
\affil{Department of Mathematics, University of Trento, Italy}

\renewcommand\Authands{ and }


\maketitle
\abstract{\noindent Most modern block ciphers are built using components whose cryptographic
strength is evaluated in terms of their resistance  to attacks on
the whole cipher.
In particular, differential properties of vectorial Boolean functions are
studied for the S-Boxes to thwart differential cryptanalysis. Little is
known on similar properties to avoid trapdoors in the design of the
block cipher. In this paper we present a form of trapdoors coming from
alternative vector space structures, which we call hidden sums, and
give a characterization on the Boolean function S-Box to avoid any such
hidden sum. We also study some properties of this new class of vectorial Boolean
functions, which we call anti-crooked, and provide a toy cipher with a hidden sum trapdoor. \\}
 


\section{Introduction}
\label{intro}
Most modern block ciphers are built using components whose cryptographic
strength is evaluated in terms of the resistance offered to attacks on
the whole cipher.
In particular, differential properties of Boolean functions are
studied for the S-Boxes to thwart differential cryptanalysis
(\cite{des,inverse}).

Little is known on similar properties to avoid trapdoors in the design of the
block cipher. In \cite{CVS} the authors investigate the minimal
properties for the S-Boxes (and the mixing layer) of an AES-like cipher 
(more precisely, a {\it translation-based cipher}, or {\it tb cipher}) 
to thwart the trapdoor coming from the imprimitivity action, first
noted in \cite{patterson}. Later, in \cite{onan} the authors provide
stronger conditions on the S-Boxes of a tb cipher that avoid  attacks
coming from any group action. This result has been generalized
to tb ciphers over any field in \cite{acds}.

In this paper we present a form of trapdoors coming from alternative vector space structure, which we call {\it hidden sums } and study from a group action point of view
in Section \ref{sec:1}. In Section \ref{sec:2} we present a class of block cipher, which is a generalization
of standard AES-like ciphers and we are able to provide a characterization on the Boolean function S-Box to avoid any such
hidden sum, which can be dangerous cryptographic trapdoors. After having studied some properties of this new class of vectorial Boolean
functions in Section \ref{sec:3}, which we call {\it anti-crooked}, in Section \ref{sec:4} we provide a toy cipher with a hidden-sum trapdoor and show that it can be thus broken with attacks which are: much faster than brute force, independent of the number of rounds, independent of the key-schedule.


\section{On hidden sums}
\label{sec:1}
Let $p\ge 2$ be a prime and $\FF_p$ be the finite field with $p$ elements. Let $V=(\mathbb{F}_p)^d$. As observed  by Li~\cite{CGC-alg-art-li03},  
the symmetric  group $\mathrm{Sym}(V)$ will contain many isomorphic copies of the affine
group $\AGL (V)$, which are its conjugates in $\mathrm{Sym}(V)$. 
So there
are  several   structures  $(V,  \circ)$  of  an
$\mathbb{F}_{p}$-vector space  on the set  $V$, where $(V,\circ)$ is the abelian additive group of the vector space. Each of these structure will yield  in general a
different copy $\mathrm{AGL}(V, \circ)$ of  the affine  group within
$\mathrm{Sym}(V)$. Any $\AGL(V,\circ)$ consists of the maps $x
\mapsto g(x) \circ v$, where $g \in \mathrm{GL}(V, \circ)$, and $v \in V$.

\begin{assumption}\label{ass:1}
 Let $G$ be a subgroup of $\mathrm{Sym}(V)$.
 
\begin{itemize}
\item  $G$  is contained in  the affine subgroup $\mathrm{AGL}(V,  \circ)$
\item  $G$ contains $T$, an  abelian regular subgroup
\end{itemize}
\end{assumption}
For the rest of this section we consider the previous assumption and we call any $\circ$
a {\it hidden sum} for $G$. Since $T$ is regular, for each $x\in V$ there exists unique $\sigma_x\in T$ such that $\sigma_x(0)=x$, therefore $T=T_V=\{\sigma_y\;|\;y\in V\}$.  
Since $T$  is abelian regular subgroup contained in  $G$,  we obtain that $T$ is an  abelian regular subgroup of $\mathrm{AGL}(V,  \circ)$. By~\cite{CGC-cry-art-cardalsal06,fcc2012}  we can define simultaneously both a structure of an associative, commutative, nilpotent ring $(V,\circ,\cdot)$ on $V$ and an operation $\Box$ on $V$ such that $(V,\Box)$ is an abelian $p$-group. The two operations are linked by
\begin{equation}
\label{eq:defbox}
x\,\Box\,y=x\circ y\circ xy
\end{equation}
and $T=T_V$ is isomorphic to $(V,\Box)$, that is $\sigma_y:x\mapsto x\,\Box\, y$.

\begin{lemma}
\label{lem:bV}
For each $u\in V$, the set $uV$ is a subgroup of $V$ with respect to both $\circ$ and  $\,\Box$.
\end{lemma}
\proof
Since the distributive property between $\circ$ and $\cdot$ holds in $(V,\circ,\cdot)$, then we have $uv_1\circ uv_2=u(v_1\circ v_2)$ for every $v_1,v_2\in V$ and so $uV$ is $\circ$-subgroup of $V$. The set $uV$ is also a $\Box$-subgroup of $V$ since
$$
uv_1\,\Box\,u v_2=u v_1\circ u v_2\circ u^2 v_1 v_2= u (v_1\circ  v_2\circ u v_1 v_2) \in uV.
$$
\endproof
Since $T<\mathrm{AGL}(V,\circ)$, for $y\in V$ there is $\kappa_y\in\mathrm{GL}(V,\circ)$ such that
\begin{equation}
\label{eq:kappa}
\sigma_y(x)=x\,\Box\, y=\kappa_y(x)\circ y
\end{equation}
for all $x\in V$.


We denote by $\boxminus t$ and $\ominus t$  the opposite of $t$ with respect to $\Box$ and $\circ$, respectively. By \eqref{eq:kappa} we can write
$$
x=x\boxminus y \Box y=\kappa_y(x\boxminus y)\circ y=\kappa_y(\kappa_{\boxminus y}(x)\circ (\boxminus y))\circ y
$$
 for all $x \in V$ and so we obtain 
\begin{equation}
\label{eq:kappa-}
\kappa_{\boxminus y} =\kappa_y^{-1}.
\end{equation}
In fact the map $(V,\Box)\to T$, $y\mapsto \gs_y$, is an isomorphism and we now show that the related map $(V,\Box)\to \GL(V,\circ)$, $y\mapsto \kappa_y$, is a group homomorphism.

\begin{proposition}
\label{lemma:kV}
For every $x,y\in V$
$$
\kappa_{x\,\Box\, y}=\kappa_y \kappa_x.
$$
Therefore $\kappa_V=\{k_x\;|\; x\in V\}$ is a $p$-group and so it acts unipotently on $(V,\circ)$.
\end{proposition}
\begin{proof}
For every $x,y,z\in V$, we have
$$
(x\,\Box\,y)\,\Box\, z=(\kappa_y(x)\circ y)\,\Box\, z=\kappa_z(\kappa_y(x))\circ \kappa_z(y)\circ z
$$
and also
$$
x\,\Box\,(y\,\Box\, z)=x\,\Box\,(\kappa_z(y)\circ z)=\kappa_{\kappa_z(y)\circ z}(x)\circ\kappa_z(y)\circ z,
$$
so from associativity we obtain 
$$
\kappa_{y\,\Box\, z}=\kappa_{\kappa_z(y)\circ z}=\kappa_z \kappa_y. 
$$
It follows that $\kappa_V$ is a group, image of the homomorphism which sends $y\in V$ to $\kappa_y\in \kappa_V$, so $\kappa_V$ is a $p$-group and thus acts  unipotently on $(V,\circ)$.

\end{proof}
By the previous proposition and some well-known results on unipotent groups (see for instance \cite{hum}), we note that there exists $y\in V\setminus\{0\}$ such that $\kappa_x(y)=y$ for each $x\in V$ and so we can consider $U=\{y\in V\;|\;x\,\Box\,y=x\circ y\mbox{ for all }x\in V\}\ne\{0\}.$ We observe that $U$ is a subgroup of both the structures $(V,\circ)$ and $(V,\Box)$.

\vspace{3mm}
\noindent Let $a\in V\setminus\{0\}$, we define the derivative of $\rho$ with respect to $a$ and $\Box$ as $D_a{(\rho,\Box)}:x\mapsto \rho(x\,\Box\, a)\boxminus \rho(x)$.
\begin{definition}
A permutation $\rho\in \mathrm{Sym}(V)$ is called anti-crooked (AC) with respect to $\Box$ if for each $a\in V\setminus\{0\}$ the set $$\mathrm{Im}(D_a{(\rho,\Box)})=\{\rho(x\,\Box\, a)\boxminus \rho(x)\,|\,x\in V\}$$ {\bf is not} a $\Box$-coset of $V$. When we say that $\rho$ is AC without specifying the sum, we mean that $\rho$ is a vectorial Boolean function which is AC with respect to the standard sum in $(\mathbb{F}_2)^{d}$.
\end{definition}

\begin{remark}
Compare the previous definition with the classical definition of crooked Boolean function \cite{crooked3}, that is, for each $a\in V\setminus\{0\}$ the set $$\mathrm{Im}(D_a{(\rho,+)})=\{\rho(x\,+\, a)+ \rho(x)\,|\,x\in V\}$$ {\bf is} a $+$-coset of $V$.
\end{remark}
Let $1_G\in G$ be the identity of $G$. We conclude this section with the main theorem linking the AC property with the existence of hidden sums with respect to $G$.
\begin{theorem}
\label{teo:AC}
If there exists $\rho\in G\setminus\{1_G\}$  AC with respect to $\Box$, then G is not contained  in any copy of the affine group in $\mathrm{Sym}(V)$.
\end{theorem}
\begin{proof}

By contradiction $G$ is
contained in  an affine subgroup $\mathrm{AGL}(V,  \circ)$ of
$\mathrm{Sym}(V)$ and there is an AC $\rho\in G$. So $G$ satisfies Assumption \ref{ass:1} and our preliminary results hold.
Since $\rho$ is $\circ$-affine and $\kappa_y\in\mathrm{GL}(V,\circ)$, by \eqref{eq:kappa-}  for $a\in U$ and $x\in V$ it follows that
$$
\begin{aligned}
\rho(x\,\Box\,a)\boxminus \rho(x)&=\rho(x)\circ\rho(a)\boxminus \rho(x)= \kappa_{\rho(x)}^{-1}(\rho(x))\circ\kappa_{\rho(x)}^{-1}(\rho(a))\circ(\boxminus \rho(x))\\
						 &=\kappa_{\rho(x)}^{-1}(\rho(a))\circ{\big(}\rho(x)\boxminus \rho(x){\big)}=\kappa_{\rho(x)}^{-1}(\rho(a))
\end{aligned}
$$
Let $W=\{\kappa_{\rho(x)}^{-1}(\rho{(a)})\,|\,x\in V\}$ be the image of $D_a{(\rho,\Box)}$. Since $\kappa_x^{-1}=\kappa_{\boxminus x}$ for every $x\in V$ and $\rho$ is a permutation we have $W=\{\kappa_{x}(\rho({a}))\,|\,x\in V\}$. For the sake of simplicity, we set $\rho(a)=\bar{a}$.
By \eqref{eq:defbox} and \eqref{eq:kappa} we have $\gk_x(\bar a)=\bar a\circ x\bar a $ and
$$
W=\{\kappa_{x}(\bar{a})\,|\,x\in V\}=\{\bar{a}\circ\bar{a}x\,|\,x\in V\}.
$$
Substituting $x$ by $\kappa_y^{-1}(x)$ in \eqref{eq:kappa} we also obtain that $x\circ y=\kappa_{y}^{-1}(x)\Box y$ for every $x,y\in V$, and so
$$
\bar{a}\circ\bar{a}{x}=\kappa_{\bar{a}}^{-1}(\bar{a}{x})\Box \bar{a}.
$$
We also have that $\kappa_{\bar{a}}^{-1}(\bar{a}{x})=\bar{a}{x}(\boxminus\bar{a})\circ \bar{a}{x}=\bar{a}({x}(\boxminus \bar{a})\circ {x})$ lies in $\bar{a}V$. Then $W=\bar{a}\,\Box\,\bar{a}V$. Hence, by Lemma \ref{lem:bV},  $\mathrm{Im}(D_a{(\rho,\Box)})$ is a $\Box$-coset of $V$ for any $a\in U$, which contradicts the assumption that $\rho$ is AC with respect to $\Box$.
\end{proof}

\section{Regularity-based block ciphers over $\mathbb{F}_{p}$}
\label{sec:2}

We introduce a block cipher  defined over  a  finite  field
$\mathbb{F}_{p}$, where the key action is not necessarily defined via the usual translations on the message space $V$, rather it can be abelian regular group acting on $V$. Indeed there are many block cipher where the keys act in a less traditional way (e.g. GOST \cite{gost}, SAFER \cite{safer}, Kalyna \cite{kalyna}).

Let $\mathcal{C}=\{ \varphi_k \mid k \in \mathcal{K} \}$ be a block cipher for which  the plaintext space
$V=(\mathbb{F}_{p})^d$, for some  $d\in\mathbb{N}$, coincides with the
ciphertext space, where any encryption function $\varphi_k\in \mathrm{Sym}(V)$ and  $\mathcal{K}$ is the
key space.
Suppose  that  any $\varphi_k$  is  the
composition  of  $l$  round   functions,  that  is, permutations
$\varphi_{k,1},\dots, \varphi _{k,l}  $, where each  $\varphi_{k,h}$
is determined by a session key $k \in \mathcal{K}$ and the round index
$h$. For each $h$ define the group generated by the $h$-round functions
\begin{equation*}
\Gamma_h(\mathcal{C})   =  \langle  \varphi_{k,h}   \mid  k\in\mathcal{K}
\rangle \le \mathrm{Sym}(V),
\end{equation*}
and the group generated by  all $\Gamma_h(\mathcal{C})$'s 
\begin{equation*}
\Gamma_\infty(\mathcal{C})  = \langle \Gamma_h(\mathcal{C})  \mid h  = 1,
\dots ,  l\rangle =  \langle \varphi_{k,h} \mid  k\in\mathcal{K}, h  = 1,
\dots, l \rangle
\end{equation*}
Clearly $\Gamma_h(\mathcal{C})\subseteq\Gamma_\infty(\mathcal{C})$ for each $h$.\\
\noindent In  the  literature, ``round''  often  refers  either  to the  ``round
index'' or to the ``round function''.


We also assume that $V$ is the Cartesian product
\begin{equation}\label{eq:directsum}
V=V_1\times\cdots\times V_n
\end{equation}
where  $n  >   1$, and the $V_i$'s are subspaces of $V$ with
$\dim_{\mathbb{F}_{p}}(V_i)=m>1$,  for  each 
$i\in\{1,\ldots,n\}$,  so  $d = m n$.  

Let $T=T_1\times\ldots\times T_n$ be an abelian  subgroup of $G=\Gamma_\infty(\mathcal{C})$ such that $T_i$ is abelian and acts regularly on $V_i$. Thanks to the results in the previous section, we have $T_i=T_{V_i}=\{\sigma_{i,v}\,|\,v\in V_i\}$ and there is an operation $\Box_i$ on $V_i$ such that $(V_i,\Box_i)$ is an abelian group and  $T_i$ is isomorphic to $(V_i,\Box_i)$. So we can define the regular action of $T$ on $V$ as the following parallel action
$$
\sigma_v(w)=w\,\Box\,v=(w_1\,\Box_1\,v_1,\ldots,w_n\,\Box_n\,v_n)
$$
where $v=(v_1,\ldots,v_n)$ and $w=(w_1,\ldots,w_n)$ with $v_i,w_i\in V_i$.

\begin{definition}
An   element   $\gamma\in\mathrm{Sym}(V)$   is   called   a   \emph{bricklayer
  transformation}  with  respect  to~\eqref{eq:directsum} if  $\gamma$
acts on  an element $v=(v_1, \cdots,  v_n)$,  with $v_i\in V_i$,
as
$$ \gamma(v)=(\gamma_1(v_1), \cdots , \gamma_n(v_n)),
$$  for  some  $\gamma_i\in\mathrm{Sym}(V_i)$.  Each $\gamma_i$  is  called  a
\emph{brick}.
\end{definition}
Clearly any $\gamma\in T$ is a bricklayer
  transformation.
Now  we   generalise  the  definition  of   translation  based  cipher
$\mathcal{C}$ over $\mathbb{F}_p$ (Definition  3.3  in~\cite{acds}) to include more general key actions.

\begin{definition}\label{deftb}
  A block cipher  $\mathcal{C} = \{ \varphi_k \mid k  \in \mathcal{K} \}$ over
  $\mathbb{F}_{p}$ is called \emph{regularity based (rb)} if
  each  encryption function $\varphi_k$  is  the  composition  of  $l$  round  functions $\varphi_{k,h}$, for $k \in \mathcal{K}$,  and $h = 1, \dots, l$, where in  turn  each round  function  $\varphi_{k,h}$  can  be written  as  a composition   $\sigma_{\phi(k,    h)} \lambda_h\gamma_h$   of   three
    permutations acting on $V$, that is,
    \begin{itemize}
    \item $\gamma_h$  is a bricklayer transformation not  depending on $k$
      and $\gamma_h(0)=0$,
    \item $\lambda_h$ is a group automorphism of $(V,\Box)$ not depending on $k$,
    \item $\phi : \mathcal{K}  \times \{1, \dots , l \} \to  V$ is the key
      scheduling function, so that  $\phi(k, h)$ is the $h$-\emph{round
        key}, given  the session key  $k$; 
      \item for at least one round index $h_0$ we have that
      the map $\mathcal{K} \to V$ given by $k \mapsto \phi(k, h_0)$ is
      surjective,  that is,  every element  of $V$  occurs as  an $h_0$-round key.
    \end{itemize}
\end{definition}

\begin{remark}
A translation based cipher is a particular case of an rb cipher when $T=(T,+)$ is the usual translation group. Note however that we drop the assumption on the properness of the mixing layer, since it has to be considered only when dealing with a primitive action.
\end{remark}
We now  work in  the group $\Gamma_h(\mathcal{C})$,  for a  fixed $h$,
omitting   for   simplicity   the   indices  $h$   for   the   various
functions. Write $\rho=\lambda\gamma$.
Since for $h_0$ the map $\mathcal{K}  \to V$ given by
$k \mapsto \phi(k, h_0)$ is surjective, note that
\begin{equation}\label{eq:h0}
\Gamma_{h_0}(\mathcal{C}) = \langle \rho, T \rangle.
\end{equation}
We are ready to prove the main result of this paper.
\begin{theorem}
If $\mathcal{C}$ is an rb cipher and $\gamma_i$ is AC with respect to $\Box_i$ for any $i\in\{1,\ldots,n\}$, then $\Gamma_{h_0}(\cC)$ and $\Gamma_\infty(\mathcal{C})$ are not contained in any copy of the affine group of $\mathrm{Sym}(V)$.
\end{theorem}
\begin{proof}
By contradiction $G=\Gamma_{h_0}(\mathcal{C})$ is contained in a copy $\mathrm{AGL}(V,\circ)$ of the affine group of $\mathrm{Sym}(V)$. By \eqref{eq:h0} we are under  Assumption \ref{ass:1} and by the proof of Theorem \ref{teo:AC} we obtain that 
if $a\ne 0$ and $a\in U\ne\{0\}$, with
$$
U=\{y\in V\;|\;x\,\Box\,y=x\circ y\mbox{ for all }x\in V\},
$$
then $\mathrm{Im}(D_a(\rho,\Box))$ is a $\Box$-coset of $V$.

Since $\lambda$ is an automorphism of $(V,\Box)$, applying $\lambda^{-1}$
to $\mathrm{Im}(D_a(\rho,\Box))$, by definition of $\Box$ and $\gamma$ we obtain
$$
\{\gamma(x\,\Box\, a)\boxminus \gamma(x)\,|\,x\in V\}=c\,\Box\, bV=(c_1\,\Box\, b_1V_1)\times\ldots\times (c_n\,\Box\, b_nV_n)
$$
for some $b=(b_1,\ldots,b_n)$ and $c=(c_1,\ldots,c_n)$ in $V$.\\
Choosing one non-zero $a\in U$, then there is a non-zero component $a_i\in V_i$ of $a$, and we have that the projection
$$
\{\gamma_i(x_i\,\Box\, a_i)\boxminus \gamma_i(x_i)\,|\,x_i\in V_i\}
$$
is a $\Box_i$-coset of $V_i$. So we obtain a contradiction since $\gamma_i$ is AC.

\end{proof}


\section{On anti crooked functions}
\label{sec:3}
In this section we consider the interesting case $p=2$, and we give some properties on the anti-crookedness of a Boolean function with respect to the usual structure $(V,+)$. 

Let $\FF=\f2$. Let $m\ge 1$, any vectorial Boolean function (vBf) $f$ from $\FF^m$ to $\FF^m$ can be expressed uniquely as a univariate polynomial in $\fm[x]$. When $f$ is also invertible we call it a vBf permutation. We denote the derivative $D_a{f}=D_a(f,+)$


\begin{definition}
Let $m,n\ge 1$. Let $f$ be a vBf from $\FF^m$ to $\FF^n$, for any $a\in\FF^m$ and $b\in \FF^n$ we define
$$
\gd_f(a,b)=|\{x\in\FF^m \,|\,D_a{f}(x)=b\}|.
$$
The \emph{differential uniformity} of $f$ is 
$$
\gd(f)=\max_{a \in\FF^m\, b\in \FF^n\\
a\neq 0}\gd_f(a,b).
$$
$f$ is said $\gd$-\emph{differential uniform} if $\gd=\gd(f)$.\\
Those functions such that $\gd(f)=2$ are said \emph{almost perfect nonlinear (APN)}.
\end{definition}

We restrict from now on to the case $m=n$, any times we write that $f$ is a vBf, we will implicit mean $f:\FF^m\to\FF^m$.

We recall the following definition presented recently in \cite{CVS}.

\begin{definition}\label{def:weakly}
Let $f$ be a vBf. $f$ is \emph{weakly-APN} if
$$
|\Im(D_a{f})|>\frac{2^{m-1}}{2}, \quad \forall\, a \in\FF^m\setminus\{0\}.
$$
\end{definition}

\noindent The notion of weakly-APN function was introduced as a necessary condition to avoid a subtle trapdoor coming from imprimitive actions (see \cite{CVS}). 

A direct check shows that an APN function is a weakly-APN. However, functions that are weakly-APN but not APN are of interest as shown in the following.
%

\begin{theorem}\label{th:pan}
Let $f$ be a vBf on $\fm$ that is weakly-APN but not APN. Then, there exists $a\in \fm$ nonzero such that $\Im(D_{a}{f})$ is not a coset of a subspace $W\subseteq \fm$. 
\end{theorem}
\begin{proof}

By contradiction suppose that for all $a\neq 0$ we have $ \Im(D_{a}{f})=w+ W$ for some $w \in \fm$ and $W$ vector space.
Since $f$ is weakly-APN, $|\Im(D_{a}{f})|$ is strictly larger than $2^{m-2}$,  thus
$|W|\geq 2^{m-1}$ and  $\dim_{\FF}(W)=m-1$. But then $D_a{f}$ is a $2$-to-$1$ function for all $a\neq 0$, which means that $f$ is APN, and this contradicts our hypothesis. In other words, there exists $a$ such that $\Im(D_{a}{f})$ is not a coset.
\end{proof}

Consider the following lemma for a power function (not necessarily a permutation).

\begin{lemma}\label{prop:2}
Let us consider $\fm$ as a vector space over $\FF$. Let $f(x)= x^d$. If there exists $a \in \fm$, $a\neq 0$, such that $\Im(D_a{f})$ is a coset of a subspace of $\fm$, then $\Im(D_{a'}{f})$ is a coset of subspace of $\fm$ for all $a'\neq 0$.
\end{lemma}
\begin{proof}
We have $\Im(D_a{f})=w+W$ where $W$ is a $\FF$-vector subspace of $\fm$ for some $w\in \fm$.
Now, let $a'\in \fm$, $a'\neq 0$, we have 
$$
D_{a'}{f}(x)=(x+a')^d+x^d=\left(\frac{a'}{a}\right)^d\left[\left(x\frac{a}{a'}+a\right)^d+\left(x\frac{a}{a'}\right)^d\right]=\left(\frac{a'}{a}\right)^dD_a{f}\left(x\frac{a}{a'}\right).
$$
So we have $\Im(D_{a'}{f})=\left( \frac{a'}{a} \right)^d \Im(D_{a}{f})=\left(\frac{a'}{a}\right)^d w+\left(\frac{a'}{a}\right)^dW=w'+W'$.
Since $W'=(a'/a)^d W$ is again an $\FF$-vector subspace of $\fm$, our claim is proved.
\end{proof}
Thanks to Lemma \ref{prop:2}, for power functions we can strengthen Theorem \ref{th:pan}.

\begin{corollary}\label{cor:w-apn}
Let $f$ be a vBf permutation on $\fm$ that is weakly-APN but not APN. If $f(x)=x^d$, then $f$ is AC.
\end{corollary}
\begin{remark}\label{rm:power}
Given an arbitrary vBf there are three possible cases: $f$ is either crooked or anti-crooked or neither. However, Lemma \ref{prop:2} shows that for a power function there are only {\bf two} possible cases: $f$ is either crooked or anti-crooked.
\end{remark}

We want now to investigate condition that guaranty the anti-crookedness of a Boolean function.

A vBf can also be represented by $m$ Boolean functions of $m$ variables, the combinations of those functions are called components. We denote by $<f,v>$ the combination corresponding to $v$. We recall the following non-linearity measure, as introduced in \cite{onwAPN}:

$$
\hat n (f):=\max_{a\in\FF^m\setminus\{0\}}|\{v\,\in\FF^m \setminus\{0\}\,:\,\deg(<D_a{f},v>)=0\}|.
$$

For all $a\in \FF^m\setminus\{0\}$, let $V_a$ be the vector space 
$$
V_a=\{v\,\in\FF^m \setminus\{0\}\,:\,\deg(<D_a{f},v>)=0\}\cup \{0\}.
$$ 
By definition, if $t=\max_{a\in\FF^m\setminus\{0\}}\dim(V_a)$, then $\hat{n}(f)=2^t-1$.
\begin{proposition}
Let $f$ be a vBf and $a\in \FF^m\setminus\{0\}$. Then $f(a)+V_a^{\perp}$ is the smallest affine subspace of $\FF^m$ containing $\Im(D_a{f})$. In particular, $\hat{n}(f)=0$ if and only if for any $a\in \FF^m\setminus\{0\}$  there is no proper affine subspace of $\FF^m$ containing $\Im(D_a{f})$.
\end{proposition}
\begin{proof}
Let $a\in\FF^m\setminus\{0\}$. Note that $V_a=\{v\,\in\FF^m\,:\,<D_a{f},v> \mbox{ is constant} \}$. Let $x\in \FF^m$, then $D_a{f}(x)=f(a)+w$, for some $w\in \FF^m$, and $<D_a{f}(x),v>=c\in\FF$ for all $v \in V_a$. In particular $c=<D_a{f}(0),v>\,=\,<f(a),v>$ and so $<w,v>=0$, that is, $w\in V_a^\perp$. Then we have $\Im(D_a{f})\subseteq f(a)+V_a^\perp$. Now, let $A$ be an affine subspace containing $\Im(D_a{f})$, then $A=f(a)+V$, for some vector subspace $V$ in $\FF^m$. For all $v\in V^{\perp}$, we have $<D_a{f},v>=<f(a),v>=c\in\FF$ and so, by definition, $V^{\perp}\subseteq V_a$. Then $A$ contains $f(a)+V_a^{\perp}$.\\
Finally,  $\hat{n}(f)=0$ if and only if $V_a=\{0\}$ for all $a\in\FF^m\setminus\{0\}$, and so our claim follows.
\end{proof}

Obviously, for any affine subspace $W$, $\Im(D_af)\not\subset W \implies \Im(D_af)\ne W$ and so we have the next corollary.
\begin{corollary}
Let $f$ be a vBf. If $\hat{n}(f)=0$ then $f$ is AC.
\end{corollary}

Coming back to power functions it is important to recall a result by Kyureghyan.
\begin{theorem}[\cite{crooked3}]
The only crooked APN power functions in $\FF_{2^n}$ are those with exponent $2^i+2^j$, $\gcd(i-j,n)=1$.
\end{theorem}
Recalling that the known exponents of APN power functions (up to factor $2^i$) are 
$$
\begin{aligned}
&2^k+1, \quad \gcd(k,m)=1\quad\mbox{(Gold's exponent \cite{gold1,gold2})}\\
&2^{2k}-2^k+1,\quad \gcd(k,m)=1\quad\mbox{(Kasami's exponent \cite{kasami})}\\
&2^{4k}+2^{3k}+2^{2k}+2^k-1,\quad m=5k \quad\mbox{(Dobbertin's functions \cite{dobbertin})}
\end{aligned}
$$
if $m=2l+1$ also
$$
\begin{aligned}
&2^l+3\quad\mbox{(Welch's exponent \cite{welch1,welch2,welch3eniho})}\\
&2^l+2^{\frac l 2}-1 \quad\mbox{if $l$ is even and}\\
&2^l+2^{\frac{3l+1}{2}}-1 \quad\mbox{if $l$ is odd} \quad\mbox{(Niho's exponent \cite{niho,welch3eniho})}\\
&2^m-2 \quad\mbox{(patched inversion \cite{inverse})}
\end{aligned}
$$
This implies that the only crooked power functions, among the known maps, are those with Gold's exponent. Thanks to Remark \ref{rm:power} we have:
\begin{corollary}
Let $x^d$ be one of the APN power functions above, with $d$ not a Gold's exponent, then $x^d$ is AC. In particular the power function $x^{2^m-2}$ is AC for all $m\ge 3$.
\end{corollary}
\begin{proof}
It follows directly from Lemma \ref{prop:2} and the theorem above. For the case of the patched inversion, from Corollary \ref{cor:w-apn}, it is AC also in even dimension.
\end{proof}

 \ \\
\indent Having examined some anti-crooked functions we would like to show some properties of this notion.
\begin{lemma}\label{rem:2.7}
If $f$ is AC then $f^{-1}$ is not necessarily AC.
\end{lemma}
\begin{proof}
 We provide an explicit example $f:\FF^{6} \to \FF^{6}$ defined by $f(x)=x^{49}$, then $f^{-1}(x)=x^5$.
A computer check shows that $f$ is anti-crooked while $f^{-1}$ is not. In particular, $\Im(D_{e^6}({f^{-1}}))$ is an affine subspace of dimension $4$, where $e$ is a primitive element of $\mathbb{F}_{64}$ such that $e^6=e^4 + e^3 + e + 1$. 
\end{proof}

We recall that two vBf's $f$ and $f'$ are called CCZ-equivalent if their graphs $G_f=\{(x,f(x)) \mid x\in\FF^n\}$ and $G_{f'}=\{(x,f'(x)) \mid x\in\FF^n\}$ are affine equivalent. We recall also that $f$ and $f'$ are called EA-equivalent if there exist three affine functions $g$, $g'$ and $g''$ such that $f'=g'\circ f\circ g+g''$.\\
Lemma \ref{rem:2.7} and the well-known fact that a vBf $f$ is CCZ-equivalent to $f^{-1}$ imply the following result.

\begin{corollary}
The anti-crookedness is not CCZ invariant.
\end{corollary} 

On the other hand, surprisingly anti-crookedness behaves well with EA invariance, as shown below.
\begin{proposition}
The anti-crookedness is EA invariant.
\end{proposition}
\begin{proof}
Let $f$ be a vBf anti-crooked, and let $g$ be a vBf such that $f$ and $g$ are EA equivalent. Then, there exist three affinities $\gl_1,\gl_2,\gl_3$ such that $g=\gl_1f\gl_2+\gl_3$. Without loss of generality we can suppose $f(0)=g(0)=0$ and $\gl_i(0)=0$ for all $i=1,2,3$. Then 
$$
\begin{aligned}
D_{a}{g}&=\gl_1 f\gl_2(x+a)+\gl_1 f\gl_2(x)+\gl_3(x+a)+\gl_3(x)\\
&=\gl_1(f(\gl_2(x)+\gl_2(a))+f(\gl_2(x))+\gl_1^{-1}\gl_3(a)),
\end{aligned}
$$
which implies 
$$
\Im(D_a{g})=\gl_1(\gl_1^{-1}\gl_3(a)\Im(D_a{f})),
$$
thus $g$ is AC if and only if $f$ is AC.
\end{proof}

\section{A block cipher with a hidden sum}
\label{sec:4}

In this section we give an example, in a small dimension, of a translation based block cipher in which it is possible to embed a hidden-sum trapdoor. The underlying field is binary and all involved functions are vBf's.

Let $m=3$, $n=2$, then $d=6$ and we have the message space $V=(\FF_2)^6$.
The mixing layer of our toy cipher is given by the matrix
$$
\gl=\left[\begin{array}{cccccc}
0 &1 &1 &0 &1 &0\\
0& 1& 0& 0& 0& 0\\
1& 1& 1& 0& 1& 0\\
0& 1& 0& 1& 1& 1\\
0& 0& 0 &0& 1& 0\\
0& 1& 0 & 1& 1& 0\end{array}\right]
$$
Note that $\gl$ is a proper mixing layer (see Definition $3.2$ \cite{acds}).
The bricklayer transformation $\gamma=(\gamma_1,\gamma_2)$ of our toy cipher is given by two identical S-boxes
$$
\gamma_1=\gamma_2=\ga^5x^6 + \ga x^5 + \ga^2x^4 + \ga^5x^3 + \ga x^2 + \ga x
$$
where $\ga$ is a primitive element of $\FF_{2^3}$ such that $\ga^3=\ga+1$.\\
The S-box $\gamma_1$ is $4$-differential uniform but it is not anti-crooked, since $\Im(D_{\ga^2}{\gamma_1})$ is an affine subspace (of dimension $1$).

%
%
%
%
%
%
%
%
%
%
%
We show now the existence of a hidden-sum trapdoor for our toy cipher.
We consider the hidden sum $\circ$ over $V_1=V_2=(\FF_2)^3$ induced by the elementary abelian regular group $T_\circ=\langle \tau_1,\tau_2,\tau_3\rangle$, where
\begin{equation}\label{eq:generatori}
\tau_1(x)=x\cdot\left[\begin{array}{ccc}
1& 0& 0\\
0 &1 &0\\
0 &1& 1\end{array}\right]+e_1,\;
\tau_2(x)=x\cdot\left[\begin{array}{ccc}
1& 0& 0\\
0& 1 &0\\
0& 0& 1\end{array}\right]+e_2,\;
\tau_3(x)=x\cdot\left[\begin{array}{ccc}
1& 1& 0\\
0 &1 &0\\
0 &0& 1\end{array}\right]+e_3,
\end{equation}
with $e_1=(1,0,0)$, $e_2=(0,1,0)$ and $e_3=(0,0,1)$. In other words, $\tau_i(x)=x\circ e_i$ for any $1\leq i\leq 3$.\\
Obviously $T=T_\circ\times T_\circ$ is an elementary abelian group inducing the hidden sum $(x_1,x_2)\circ'(y_1,y_2)=(x_1\circ y_1,x_2\circ y_2)$ on $V=V_1\times V_2$. We denote\\
 $e_1'=(1,0,0,0,0,0),\dots,e_6'=(0,0,0,0,0,1)$ and we claim the following.

\begin{theorem}\label{th:5}
$\langle T_+, \gl \gamma\rangle \subseteq \AGL(V,\circ')$, where $T_+$ is the translation group with respect to $+$.
\end{theorem}
\begin{proof}
By a computer check it results that $\gs_{e_i'}\in \AGL(V,\circ')$ for all  $1\le i\le 6$ where $\gs_{e_i'}(x)=x+e_i'$. Since these generate $T_+$, this implies $T_+\subseteq \AGL(V,\circ')$.\\
By another computer check the map $\gl\gamma$ lies in $\AGL(V,\circ')$
\end{proof}
Thanks to the previous theorem, $\circ'$ is a hidden sum for our toy cipher, but it remains to verify whether it is possible to use it to attack the toy cipher with an attack that costs less than brute force. 
We have not discussed the key schedule and the number of rounds yet. We have in mind a cipher where the number of rounds is so large to make any classical attack useless (such as differential cryptanalysis) and the key scheduling offer no weakness.  Therefore, the hidden sum will  actually be essential to break the cipher only if the attack that we build will cost significantly less than $64$ encryptions, considering that the key space is $(\mathbb{F}_2)^6$.
\begin{remark}
Given a sum $\Box$, the vectors $e_1,e_2,e_3$ may not be a linear basis of $(V_1,\Box)$. For this specific sum $\circ$, the vectors $e_1,e_2,e_3$ actually form a basis for $(V_1,\circ)$ as can be checked by computer. Let $x=(x_1,x_2,x_3)\in V_1$, from \eqref{eq:generatori} we can simply write
 {\small
$$
 \tau_1(x)=(x_1+1,x_2+x_3,x_3),\tau_2(x)=(x_1,x_2+1,x_3),\tau_3(x)=(x_1,x_1+x_2,x_3+1).
 $$}
 \noindent Let us write $x$ as a linear combination of $e_1$, $e_2$ and $e_3$ w.r.t. to the sum $\circ$, i.e. $x=\gl_1 e_1\circ \gl_2 e_2 \circ \gl_3 e_3$. We claim that $\lambda_1=x_1$, $\lambda_3=x_3$ and $\lambda_2=\lambda_1\lambda_3+x_2$.
To show that, we adopt the following notation:  if $b\in \mathbb{F}_2$, we can write $\tau_v^b(x)$ denoting either $\tau_v(x)$ (when $b=1$) or $x$ (when $b=0$). Then
{\small $$
\begin{aligned}
x&=(\gl_1 e_1\circ \gl_2e_2 )\circ \gl_3 e_3=\tau_3^{\lambda_3}(\gl_1 e_1\circ \gl_2e_2)=\tau_3^{\gl_3}(\tau_2^{\gl_2}(\lambda_1 e_1))=\tau_2^{\gl_2}(\tau_3^{\gl_3}(\lambda_1 e_1))\\
&=\tau_2^{\gl_2}(\tau_3^{\gl_3}((\gl_1,0,0)))
=\tau_2^{\gl_2}((\gl_1,\gl_3\gl_1,\gl_3))=(\gl_1,\gl_1\gl_3+\gl_2,\gl_3).
\end{aligned}
$$}
So 
$$
(x_1,x_2,x_3)=x=(\gl_1,\gl_1\gl_3+\gl_2,\gl_3)
$$
and our claim is proved.
\end{remark}

Thanks to the previous remark we can find the coefficients of  a vector $v'=(v,u)\in V$ with respect to $\circ'$ by using the following algorithm separately on the two bricks of $w$.
\begin{algorithm}\label{alg:1}
\ \\
{\bf Input:} vector $x\in \mathbb{F}_2^3$\\
{\bf Output:} coefficients $\gl_1$, $\gl_2$ and $\gl_3$.\\
$[1]$ $\lambda_1\leftarrow x_1$;\\
$[2]$ $\lambda_3\leftarrow x_3$;\\
$[3]$ $\gl_2\leftarrow\gl_1\gl_3+x_2$;\\
return $\gl_1,\gl_2,\gl_3$.
\end{algorithm}

Let $v'=(v,u)\in V$, we write 
$$v=\gl_1^v  e_1\circ \gl_2^v e_2 \circ \gl_3^v e_3\mbox{ and }u=\gl_1^u  e_1\circ \gl_2^u e_2 \circ \gl_3^u e_3.
$$ 
We denote by 
$$
[v']=[\gl_1^{v},\gl_2^{v},\gl_3^{v},\gl_1^{u},\gl_2^{u},\gl_3^{u}]
$$
 the vector with the coefficients obtained from the bricks of $v$ using Algorithm~\ref{alg:1}. 
 
Let $\varphi=\varphi_k$ be the encryption function, with a given unknown session key $k$. We want to mount two attacks by computing the matrix $M$ and the translation vector $t$ defining $\varphi\in\AGL(V,\circ')$, which exist thanks to Theorem \ref{th:5}.\\
Assume we can call the encryption oracle. Then $M$ can be computed from the $7$ ciphertexts $\varphi(0),\varphi(e_1'),\dots,\varphi(e_6')$, since the translation vector is $[t]=[\varphi(0)]$ and the $[\varphi(e_i')]+[t]$'s represent the matrix rows. In other words, we will have 
 $$
 [\varphi(v')]=[v']\cdot M+[t],\quad [\varphi^{-1}(v')]=([v']+[t])\cdot M^{-1},
 $$
for all $v'\in V$, where the product row by column is the standard scalar product. The knowledge of  $M$ and $M^{-1}$ provides a global deduction (reconstruction), since it becomes trivial to encrypt and decrypt. However, we have an alternative depending on how we compute $\varphi^{-1}$:
\begin{itemize}
\item if we compute $M^{-1}$ from $M$, by  applying for example Gaussian reduction, we will need only our $7$ initial encryptions;
\item else we can compute $M^{-1}$ from the action of $\phi^{-1}$, assuming we can call the decryption oracle, simply by performing the $7$ decryptions $\varphi^{-1}(e_i')$ and $\varphi^{-1}(0)$; indeed, the rows of $M^{-1}$ will obviously be $[\varphi^{-1}(e_i')]+[\varphi^{-1}(0)]$.
\end{itemize}
The first attack requires more binary operations, since we need a matrix inversion, but only $7$ encryptions. The second attack requires both $7$ encryptions and $7$ decryptions, but less binary operations.  The first attack is  a chosen-plaintext attack, while the second is a chosen-plaintext/chosen-ciphertext attack. Both obtain the same goal, that is, the complete reconstruction of the encryption and decryption functions. Note that, since an encryption/decryption will cost a huge number of binary operations in our assumptions (we are supposing that many rounds are present), the first attack is more dangerous and its cost is approximately that of $7$ encryptions, while the cost of the second attack is approximately $14$ encryptions (being the cost of an encryption  close to the cost of a decryption).









\end{document}